\documentclass[12pt, letterpaper]{article}
\usepackage[utf8]{inputenc}
\usepackage[english]{babel}
\usepackage{amsmath}
\usepackage{amsfonts}
\usepackage{amssymb}
\usepackage{amsthm}
\usepackage{graphicx}
\usepackage[left=1in,top=1in,right=1in,bottom=1in]{geometry}
\usepackage{caption}

\newtheorem{teo}{Theorem}

\newtheorem{prop}{Proposition}
\newtheorem{cor}{Corollary}
\newtheorem*{remark}{Remark}
\newtheorem{con}{Condition}
\newcommand{\E}{\mathbb{E}}

\numberwithin{equation}{section}
\begin{document}

\title{Extremal Laws for Laplacian Random Matrices \thanks{%
Research supported by Conacyt Grant A1-S-976 }}
\author{Santiago Arenas-Velilla\thanks{%
CIMAT, Guanajuato, Mexico, \texttt{santiago.arenas@cimat.mx}}\qquad Victor P%
\'{e}rez-Abreu\thanks{%
Instituto Polit\'{e}cnico Nacional, Mexico, \texttt{vpereza@ipn.mx}}}
\maketitle

\begin{abstract}
For an $n\times n$ Laplacian random matrix $L$ with Gaussian entries it is
proven that the fluctuations of the largest eigenvalue and the largest
diagonal entry of $L/\sqrt{n-1}$ are Gumbel. We first establish suitable
non-asymptotic estimates and bounds for the largest eigenvalue of $L$ in
terms of the largest diagonal element of $L$. An expository review of
existing results for the asymptotic spectrum of a Laplacian random matrix is
also presented, with the goal of noting the differences from the
corresponding classical results for Wigner random matrices. Extensions to
Laplacian block random matrices are indicated.
\end{abstract}

\section{Introduction}

The largest eigenvalue of a random Laplacian matrix plays an important role in applications  of complex graphs \cite{chung2006complex}, analysis of algorithms for the $\mathbb{Z}_2$ Synchronization, and community detection in the Stochastic Block Model  \cite{abbe2014decoding,goemans1995improved}, 
among other optimization problems in semidefinite programming \cite{bandeira2015random}. Its limiting behavior both
almost surely and in probability has been considered by Bryc, Dembo and
Jiang \cite{BDJ06} and Ding and Jiang \cite{ding2010spectral}. The present
paper considers the limiting law of the largest eigenvalue of a random
Laplacian matrix.

The famous Tracy--Widom distribution appears in random matrix
theory as the limiting law of the largest eigenvalue for a large class of
random matrices, including classical Gaussian \cite{tracy1994level}, \cite%
{tracy1996orthogonal}, Wigner \cite{soshnikov1999universality}, Wishart \cite%
{johnstone2001distribution}, and sparse matrices \cite{lee2018local}.

However, there are examples of structured random matrices constructed from a
number of independent random variables less than the one used to construct Wigner matrices, and for which the limiting law of the largest eigenvalue is
Gumbel. This is the case for palindromic circulant random matrices, as shown
by Bryc and Sethuraman \cite{bryc2009remark}, and for Gaussian Hermitian
circulant matrices, as shown by Meckes \cite{meckes2009some}. The Gumbel
distribution is also the limiting law of the spectral radius of Ginibre
ensembles \cite{rider2014extremal}.

One of the purposes of the present paper is to show that the Gumbel
distribution is also the limiting law of the largest eigenvalue of a
Laplacian matrix constructed from independent real Gaussian random variables. This kind of random matrix appears in the $\mathbb{Z}_2$ Synchronization problem of recovering binary labels with Gaussian noise \cite{bandeira2015random}.

More specifically, let $X=(X_{ij})$ be an $n\times n$ symmetric matrix.
 The \textit{Laplacian matrix} $L=L_{X} =(L_{ij})$ of $X$ is
defined by

\begin{equation}
L=D_{X}-X  \label{laplacian}
\end{equation}

\noindent{}where $D_X$ is the diagonal matrix whose entries are given by

\begin{equation}
(D_{X})_{ii}=\sum_{j=1}^{n}X_{ij}.  \label{diagonal}
\end{equation}

Matrices whose row sums are zero are often called Markov matrices, so a Laplacian
matrix is an example of a Markov matrix. 

Let $X=(X_{ij})$ be a \textit{Wigner random matrix}, that is, an $n\times n$ symmetric
random matrix whose off-diagonal entries $\{X_{ij},1\leq i<j\leq n\}$ are real-valued, independent and identically distributed random variables with zero mean and
unit variance, and the diagonal entries $\{X_{ii},1\leq i\leq n\}$ are
independent random variables with zero mean and variance two that are
independent of the off-diagonal entries.

We observe that for a Wigner matrix $X$, the corresponding Laplacian matrix $L$ has $%
n(n-1)/2 $ independent random variables and its diagonal entries are
not independent of the non-diagonal entries, while $X$
has $n(n+1)/2$ independent random variables.

Let $\lambda _{\max }(X)$ denote the largest
eigenvalue of $X$ and $\lambda _{\max }(L)$ denote that of $L$. It is well known that under suitable
moment conditions, the limiting spectral distribution of $X/\sqrt{n}$ is the
semicircle law in $[-2,2]$, that $\lambda _{\max }(X)/\sqrt{n}$ converges to 
$2$ almost surely as $n$ goes to infinity, and that the limiting law of $%
\lambda _{\max }(X)$, after appropriate rescaling, obeys the
Tracy--Widom distribution.

In contrast, for the Laplacian matrix $L$, under suitable conditions on the
entries of $X$, the limiting spectral law of $L/\sqrt{n}$ is a distribution
with unbounded support which is the free convolution of the semicircular law
with the standard Gaussian distribution \cite{BDJ06}, \cite{ding2010spectral}%
. Moreover, $\lambda _{\max }(L)/\sqrt{n\log n}$ converges to $\sqrt{2}$ in
probability as $n$ goes to infinity \cite{ding2010spectral}. In Section 2 we
provide a brief summary of the above results for Laplacian random matrices,
and also indicate some extensions to block matrices.

Our main result is the weak convergence of $\lambda _{\max }(L)/\sqrt{n-1}$
with an appropriate rescaling to the Gumbel distribution, as
presented in Theorem \ref{maintheorem}. In order to prove this result, we
first show in Section 3 that the limiting law of the largest entry $%
\max_{i}L_{ii}/\sqrt{n-1}$ of the sequence of the triangular array $%
\{L_{ii}=L_{ii}^{(n)}:i=1,\ldots ,n\}$ is a Gumbel distribution for which we
use the form of the non-stationary covariance matrix of $\{L_{ii}\}$ to
produce a useful approximating sequence of independent Gaussian random
variables. Then we present in Section 4 suitable non-asymptotic estimates and
bounds for the largest eigenvalue of $L$ in terms of the largest diagonal
element of $L$ for a class of Laplacian matrices, namely, those such that $\max_{i}L_{ii}/\sqrt{n-1}$ grows faster than $\sqrt{\log n}$.  Finally, we prove in Section 5 that, after appropriate rescaling, $\max_{i}L_{ii}/\sqrt{n-1}$ and $\lambda _{\max }(L)/%
\sqrt{n-1}$ have the same limiting law under Gaussian assumptions.

\section{On the asymptotic spectrum of the Laplacian}

In this section we present a summary of two kinds of known results 
for the asymptotic spectrum of a random Laplacian matrix,
namely, the weak convergence of the empirical spectral distribution of
the Laplacian matrix and the almost sure behavior of its largest eigenvalue.
We remark on the differences from the corresponding results for Wigner matrices and
also indicate some new useful extensions to block matrices.

\subsection{Weak convergence of the empirical spectral distribution}

As is by now well known, the empirical spectral distribution plays an
important role in the theory of large dimensional random matrices. For a
symmetric $n\times n$ matrix $A_{n}$, let $\lambda _{1}(A_{n})\geq \cdot
\cdot \cdot \geq \lambda _{n}(A_{n})$ denote the eigenvalues of the matrix $%
A_{n}$ and let $\tilde{F}_{n}(x)=\frac{1}{n}\sum_{i=1}^{n}\mathbb{I}\left\{
\lambda _{i}(A_{n})\leq x\right\} ,x\in \mathbb{R}$ be the empirical
distribution of its eigenvalues. When $A_{n}$ is a random matrix, $\tilde{F}%
_{n}$ is a random distribution function in $\mathbb{R}.$

In his pioneering work, Wigner \cite{wigner1958distribution} proved that for
a Wigner matrix $X$, the empirical spectral
distribution of $\frac{1}{\sqrt{n}}X$ converges weakly to the semicircle law 
$S$ in $[-2,2]$ 
\begin{equation*}
S(dx)=\frac{1}{2\pi }\sqrt{4-x^{2}}\mathbb{I}_{|x|\leq 2}dx.
\end{equation*}
\noindent with probability one 

In the case when $X_{n}$ is a Wigner matrix with Gaussian entries and the
sequence $(X_{n})_{n\geq 1}$ is independent of a sequence of diagonal random
matrices $(\widetilde{D}_{n})_{n\geq 1}$ also with independent Gaussian
entries of zero mean, Pastur and Vasilchuk \cite{PaVa00} proved that with
probability one the empirical spectral distribution of $\frac{1}{\sqrt{n}}(%
\widetilde{D}_{n}-X_{n})$ converges weakly to a non-random distribution $%
\gamma _{M}$, which is the free convolution $S\boxplus \gamma $ of the
semicircle law $S$ in $(-2,2)$ with the standard Gaussian distribution $\gamma $.

This result is a consequence of a pioneering result by Voiculescu, who
established that, under suitable conditions, if $A=(A_{n})_{n\geq 1}$ and $%
B=(B_{n})_{n\geq 1}$ are independent sequences of $n\times n$ random
matrices with empirical spectral distributions $F_{A}$ and $F_{B}$, then $A$
and $B$ become asymptotically free and the asymptotic spectral distribution
of the sum $(A_{n}+B_{n})_{n\geq 1}$ is the free convolution $F_{A}\boxplus
F_{B}.$ This remarkable fact shows the usefulness of free probability
in large dimensional random matrices, since from a knowledge of the eigenvalues of $%
A_{n} $ and $B_{n}$, it is not easy to find the eigenvalues of the sum $%
A_{n}+B_{n}$, unless $A_{n}$ and $B_{n}$ commute. We refer to the books 
\cite{AGZ10} and \cite{MiSp17} for the definition and properties of the free
convolution $\boxplus $ and the asymptotic freeness of random matrices, and to 
\cite{Bi97} for a study of distributions that are free convolutions with
the semicircle distribution.

A natural question is whether it is possible to find a corresponding
spectral asymptotic result for the Laplacian matrix $L_{n}$ given by (\ref%
{laplacian}). The challenge is that the elements of the diagonal $D_{n}$ of $%
L_{n}$ are not independent of the non-diagonal entries of $X_{n}$. However,
the result in \cite{PaVa00} predicted that the result holds also for $L_{n}$%
, but the problem is nontrivial because $D_{n}$ strongly depends on $X_{n}$.
This was done by Bryc, Dembo and Jiang \cite{BDJ06} in the context of Markov
random matrices and when the entries $X_{ij}$ are independent and
identically distributed random variables. We summarize some of their results
in what follows.

\begin{teo}
\label{DEElaplaciana} Let $L_{n}=L_{X_{n}}$ be the Laplacian matrix of $%
X_{n}=(X_{ij}^{(n)})$ where $\left\{ X_{ij}^{(n)}:1\leq i<j\leq n \right\}
_{n\geq 1}$ is a collection of independent identically distributed random
variables of mean $m$ and finite variance $\sigma ^{2}$. Let $\tilde{F}_{n}$
be the empirical spectral distribution of $\frac{1}{\sqrt{n}}L_{n}.$

i) If $m=0$ and $\sigma ^{2}=1$, then, with probability one, as $n\rightarrow
\infty $, $\tilde{F}_{n}$ converges weakly to the free convolution $\gamma
_{M}=S\boxplus \gamma $ of the semicircle law $S$ in $[-2,2]$ with the standard
Gaussian distribution $\gamma .$ 

ii) The measure $\gamma _{M}$ is a nonrandom symmetric probability measure
with unbounded support which does not depend on the distribution of $%
X_{ij}^{(n)}$.

iii) If $m$ is not zero and $\sigma ^{2}<\infty $, then $\tilde{F}_{n}$
converges weakly to the degenerate distribution $\delta _{-m}$ as $%
n\rightarrow \infty $.
\end{teo}

The above result was extended by Ding and Jiang \cite{ding2010spectral} to
the case when the random variables $X_{ij}^{(n)}$ are independent but not
necessarily identically distributed. More specifically, they make the
following assumption.

\begin{con}
Let $X_n = (X_{ij}^{(n)})$ be an $n \times n$ symmetric matrix with $%
\{X_{ij}^{(n)};1\leq i< j\leq n, n\geq2\}$ random variables defined on
the same probability space and $\{X_{ij}^{(n)};1\leq i< j\leq n\}$ 
independent for each $n \geq 2$ (not necessarily identically distributed)
with $X_{ij}^{(n)}=X_{ji}^{(n)}$, $\mathbb{E}(X _{ij}^{(n)})=\mu_n$, $\text{%
Var}(X_{ij}^{(n)})=\sigma_n^2>0$ for all $1\leq i< j\leq n$ $n\geq2$, and 
\begin{equation*}
\sup_{1\leq i<j\leq n, n\geq2}\mathbb{E}|(X_{ij}^{(n)}-\mu_n)/\sigma
_n|^{p}<\infty
\end{equation*}
for some $p>0$.
\end{con}

Assuming this condition for some $p>4,$ it is proved in \cite%
{ding2010spectral} that if 
\begin{equation*}
\tilde{F}_{n}(x)=\frac{1}{n}\sum_{i=1}^{n}\mathbb{I}\left\{ \frac{\lambda
_{i}(L_{n})-n\mu _{n}}{\sqrt{n}\sigma _{n}}\leq x\right\} ,\qquad x\in 
\mathbb{R},
\end{equation*}%
then, as $n\rightarrow \infty $, with probability one, $\tilde{F}_{n}$
converges weakly to the distribution of the free convolution $\gamma _{M}=S\boxplus \gamma $ of
the semicircle law $S$ in $[-2,2]$ with the standard Gaussian distribution $%
\gamma $.

\subsection{On the distribution of $\protect\gamma _{M}=S\boxplus \protect%
\gamma $}

Several other properties of the limiting spectral distribution $\gamma
_{M}=S\boxplus \gamma $ are known. In addition to being symmetric and with
unbounded support, it is also shown in \cite{BDJ06} that this distribution
has a smooth bounded density $f_{\gamma _{M}}$ and is determined by
its moments. Some of these properties are based on \cite{Bi97},
where the author consider distributions that are free convolutions with the semicircle
distribution. It is also proved in \cite{BDJ06} that $\gamma _{M}$ has even
moments given by the combinatorial formula

\begin{equation*}
m_{2k}=\sum_{\omega :\left\vert \omega \right\vert =2k}2^{h(\omega )}
\end{equation*}%
where $h(\omega )$ is the so-called height function that assigns to a pair partition $%
\omega $ the number of connected blocks which are of cardinality 2.

However, an explicit expression for $\gamma _{M}$ is not known.
One can easily simulate this distribution using large dimensional random
matrices and predict the form of the empirical spectral distribution. In
fact, Figure 1 presents an example of a consistent result of a large simulation study we
have done, finding that the density $f_{\gamma _{M}}$ can be statistically
approximated by a mixture of a Gaussian distribution and a semicircle
distribution, although, in this approximation, the
derivative of the density has two singularities. 

More specifically, if a Laplacian matrix $L$ is constructed from independent
identically distributed random variables $\{X_{ij}^{(n)};1\leq i<j\leq
n,n\geq 2\},$ with zero mean and variance $\sigma ^{2},$ a good approximation
$\widehat{f}_{\gamma _{M}}$ is 

\begin{align}
\widehat{f}_{\gamma _{M}}(x)& =\frac{\alpha }{2\sigma \sqrt{2}}f_{S}\left( 
\frac{x}{\sigma \sqrt{2}}\right) +\frac{1-\alpha }{2\sigma \sqrt{2}}%
f_{\gamma }\left( \frac{x}{\sigma \sqrt{2}}\right)  \\
& =\frac{\alpha }{2\sigma \sqrt{2}}f_{S_{\sigma }}\left( x\right) +\frac{%
1-\alpha }{2\sigma \sqrt{2}}f_{\gamma _{\sigma }}\left( x\right) ,\text{ \ \ 
}x\in \mathbb{R},
\end{align}%
where $S_{\sigma }$ is the semicircle distribution on $[-\sqrt{2}\sigma ,%
\sqrt{2}\sigma ]$, $\gamma _{\sigma }$ is the Gaussian distribution of zero
mean and variance $2\sigma ^{2}$, and $\alpha =\sqrt{2}/2$ is the weight of
the mixture. 

\begin{figure}[h]\label{figura1}
\centering
\includegraphics[scale=1]{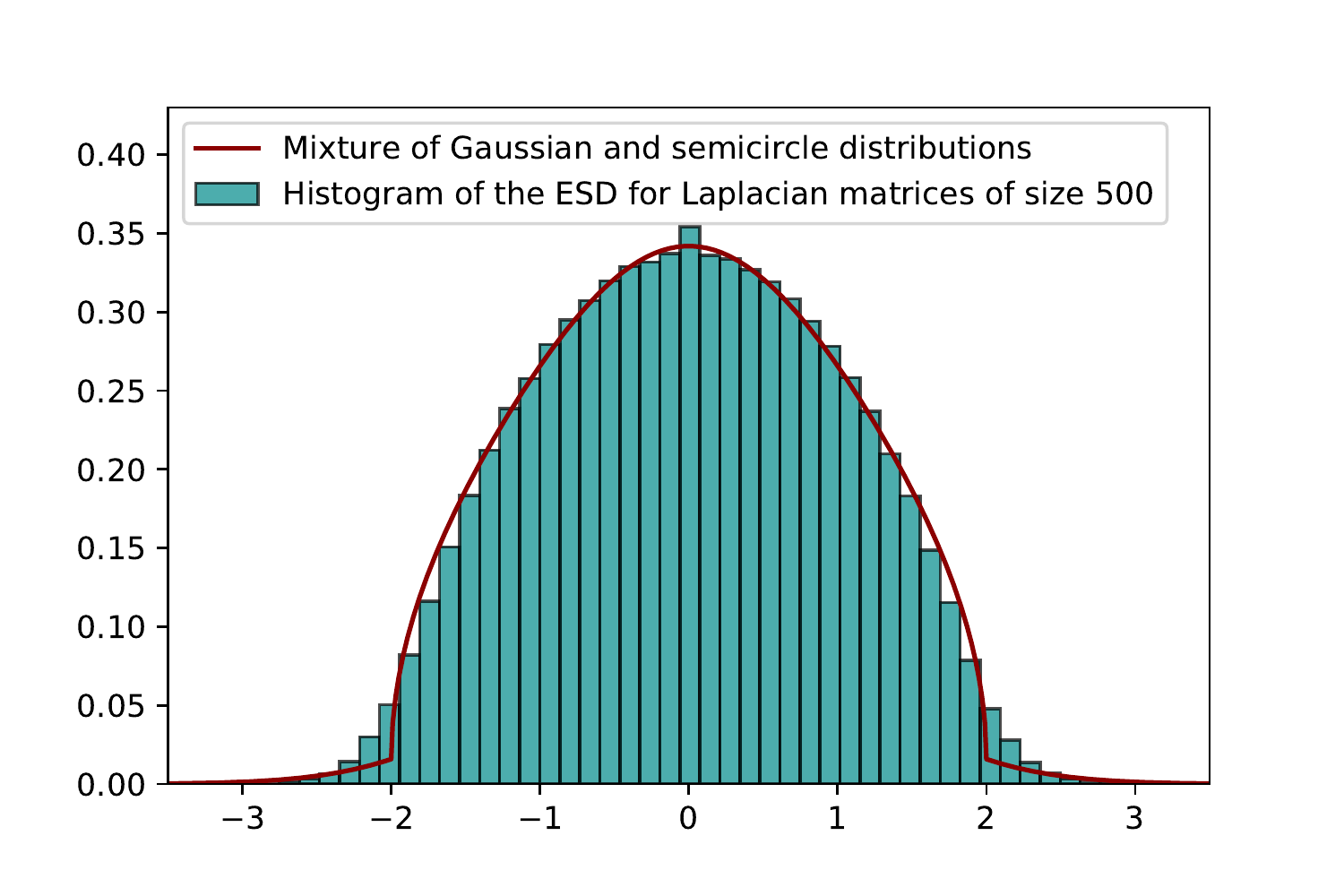}
\captionsetup{labelformat=empty}
\caption{Figure 1}
\end{figure}

\subsection{Almost sure behavior of the largest eigenvalue}

Bai and Yin \cite{bai1988necessary} proved that the largest eigenvalue $%
\lambda _{\max }(X)$ of a Wigner matrix $X$ satisfies   
\begin{equation}
\lim_{n\rightarrow \infty }\frac{1}{\sqrt{n}}\lambda _{\max }(X)=2\quad
\text{a.s.}\quad \text{ as }n\rightarrow \infty .  \label{mayoreigWigner}
\end{equation}

The next result from Ding and Jiang \cite{ding2010spectral} gives the same
almost sure convergence for the largest eigenvalue for symmetric
matrices as in the Wigner case, but with unit variance in the diagonal $X$.
The second part is the same almost sure convergence of the maximum eigenvalue of symmetric matrices with zero
diagonal elements as is the case with adjacency matrices in random graphs. 

\begin{prop}
\label{lemaDJ} Let $\mathbf{U}_{n}=(u_{ij}^{(n)})$ be an $n\times n$
symmetric random matrix and assume that for each $n\geq 1$, $%
\{u_{ij}^{(n)}:1\leq i\leq j\leq n\}$ are independent random variables with $%
\mathbb{E}u_{ij}^{(n)}=0$, $\text{Var}(u_{ij}^{(n)})=1$ for all $1\leq
i,j\leq n$, and $\sup_{1\leq i,j,\leq n,n\geq 1}\mathbb{E}%
|u_{ij}^{(n)}|^{6+\delta }<\infty $ for some $\delta >0$. Then:

\begin{enumerate}
\item[(i)] 
\begin{equation*}
\lim_{n \to \infty} \frac{\lambda_{\max}(\mathbf{U}_n)}{\sqrt{n}} = 2 \quad
\text{a.s.}
\end{equation*}

\item[(ii)] The statement in (i) still holds if $\mathbf{U}_{n}$ is replaced
by $\mathbf{U}_{n}-\text{diag}(u_{ij}^{(n)})$.
\end{enumerate}
\end{prop}

However, for Laplacian matrices, the asymptotic empirical spectral
distribution has unbounded support and therefore it is not expected to have
the same type of convergence to a finite limit, if at all, with the same
normalization $1/\sqrt{n}$. A description of the asymptotic behavior of the largest eigenvalue of Laplacian matrices is
also given in \cite[Theorem 1]{ding2010spectral}. It turns out that when normalized by $1/%
\sqrt{n\log n}$, the largest eigenvalue of Laplacian matrices converges to a
finite limit in probability.

\begin{teo}
\label{TeoDJ} Let $L_{n}=L_{X_{n}}$ be the Laplacian matrix of an $n\times n$
symmetric matrix $X_{n}$ whose entries satisfy Condition 1 for some $p>6$%
. If $\mu _{n}=0$ and $\sigma _{n}=1$ for all $n\geq 2$, then

\begin{enumerate}
\item[(a)] 
\begin{equation*}
\frac{\lambda_{\max} (L_{n})}{\sqrt{n\log n}} \to \sqrt{2}
\end{equation*}
in probability as $n\to \infty$.
\end{enumerate}

\begin{enumerate}
\item[(b)] Furthermore, if $\{L_{2},L_{3},\ldots \}$ are independent Laplacian matrices, then

\begin{equation*}
\liminf_{n\rightarrow \infty }\frac{\lambda _{\max }(L_{n})}{\sqrt{n\log n}}=%
\sqrt{2}\quad \text{a.s.}
\end{equation*}%
and 
\begin{equation*}
\limsup_{n\rightarrow \infty }\frac{\lambda _{\max }(L_{n})}{\sqrt{n\log n}}%
=2\quad \text{a.s.}
\end{equation*}%
\noindent {}and the sequence $\{\lambda _{\max }(L_{n})/\sqrt{n\log n};n\geq
2\}$ is dense in $[\sqrt{2},2]\quad \text{a.s.}$
\end{enumerate}
\end{teo}

As a consequence of Theorem \ref{TeoDJ}, we have the following result (which
will be used in Section 4), giving upper and lower bounds for the
largest eigenvalue of a Laplacian matrix.

We say that an event $\mathcal{C}_{n}$ occurs \textit{with high probability} ($\mathrm{w.h.p}$) if for
every $\eta >0$, there exists an $n_{0}\in \mathbb{N}$ such that

\begin{equation*}
\mathbb{P}(\mathcal{C}_n) \geq 1- \eta \qquad \text{for } \quad n > n_{0}.
\end{equation*}

\begin{cor}
\label{corDJ} Let $L=L_{X}$ be the Laplacian matrix of an $n\times n$ symmetric
matrix $X$ whose entries satisfy Condition 1 for some $p>6$. If $\mu
_{n}=0$ and $\sigma _{n}=1$ for all $n\geq 2$, then for all $\epsilon >0$

\begin{equation}
\lambda_{\max}(L) \leq \sqrt{(2+\epsilon) n \log n}
\end{equation}

\noindent with high probability,
and

\begin{equation}  \label{cotainferior}
\lambda_{\max}(L) \geq (2 \sqrt{2} - \sqrt{2+ \epsilon} ) \sqrt{n \log n}.
\end{equation}
\end{cor}

\begin{proof}
This follows from Theorem \ref{TeoDJ}, which ensures that

$$\mathbb{P}\left( \left| \frac{\lambda_{\max}(L)}{\sqrt{n \log n}} - \sqrt{2} \right| > \delta \right) \to 0 \qquad \text{as} \quad n\to \infty$$

\noindent{}taking $\delta = \sqrt{2+ \epsilon} - \sqrt{2} >0$.
\end{proof}
Observe that the right side of (\ref{cotainferior}) is non-negative if $0
< \epsilon < 6$.

\subsection{Block Laplacian matrices}

We now indicate some extensions of Theorem \ref{TeoDJ} to certain block
diagonal Laplacian matrices 
\begin{equation}\label{blockLap}
L=\left[ 
\begin{array}{cccc}
L_{1} & 0 & \ldots & 0 \\ 
0 & L_{2} & \ldots & 0 \\ 
\vdots & \vdots & \ddots & \vdots \\ 
0 & 0 & \cdots & L_{k}%
\end{array}%
\right] ,
\end{equation}

\noindent {}where $\{L_{1},L_{2},\ldots ,L_{k}\}$ are independent Laplacian random 
square matrices of the same size. These block matrices and the behavior of the
largest eigenvalue of $L$ are important in the optimization problems of Stochastic
Block Models: see \cite{bandeira2015random}.

As an application of Theorem \ref{TeoDJ}, we find the asymptotic convergence
of the largest eigenvalue for block diagonal Laplacian matrices.

\begin{prop}
\label{ExtensionBloques}

Let $L$ be a $k$-block Laplacian diagonal matrix in which each block is of
size $n/k$, and the off-diagonal entries in each block are independent.
Suppose that Condition 1 holds for some $p>6$. If whenever $L_{ij}$ is a nonzero entry of $L$, $\mathbb{E}[L_{ij}]=0$ and 
$\text{Var}(L_{ij})=1$ for all $n\geq 2$, then

\begin{equation}
\frac{\lambda_{\max}(L)}{\sqrt{n \log n}} \to \sqrt{\frac{2}{k}}
\end{equation}
in probability as $n\to \infty$.
\end{prop}

\begin{proof}
Let $L_1, L_2, \ldots, L_k$ be the  $k$ blocks of $L$. For $i =1, \ldots, k$ let
$$\Omega_{i}^{(n)} = \{ \omega \in \Omega : \lambda_{\max}(L_j) \leq \lambda_{\max} (L_i) \text{ for all } j  \}.$$
Hence
\begin{align*}
\mathbb{P} \left(\left| \frac{\lambda_{\max}(L)}{\sqrt{n \log n}} -\sqrt{\frac{2}{k}} \right| > \epsilon \right) &=  \sum_{i=1}^{k}\mathbb{P} \left( \left\{\left| \frac{\lambda_{\max}(L)}{\sqrt{n \log n}} -\sqrt{\frac{2}{k}} \right| > \epsilon  \right\} \bigcap \Omega_{i}^{(n)}\right) \\
&\leq \sum_{i=1}^{k} \mathbb{P} \left(\left| \frac{\lambda_{\max}(L_{i})}{\sqrt{n \log n}} -\sqrt{\frac{2}{k}} \right| > \epsilon \right)  \\
&\leq \sum_{i=1}^{k}\mathbb{P} \left(\left| \frac{\lambda_{\max}(L_{i})}{\sqrt{\frac{n}{k} \log \frac{n}{k}}} \frac{\sqrt{\frac{n}{k} \log \frac{n}{k}}}{\sqrt{n \log n}}  -\sqrt{\frac{2}{k}} \right| > \epsilon \right)
\end{align*}

Note that each matrix $L_i$ has size $n/k \times n/k$ and satisfies the hypotheses of Theorem \ref{TeoDJ}, so, for $i = 1, \ldots, k$, 

$$ \frac{\lambda_{\max}(L_i)}{\sqrt{\frac{n}{k}\log{\frac{n}{k}}}} \to \sqrt{2} $$
\noindent{}in probability as $n \to \infty$  and 

$$\frac{\sqrt{\frac{n}{k} \log \frac{n}{k}}}{\sqrt{n \log n}}   \to \frac{1}{\sqrt{k}} \qquad \text{as} \qquad n \to \infty.$$

\noindent{}Then, by Slutsky's theorem, we have that for each $i = 1, \ldots , k$, 
$$\frac{\lambda_{\max}(L_i)}{\sqrt{\frac{n}{k}\log{\frac{n}{k}}}}\frac{\sqrt{\frac{n}{k} \log \frac{n}{k}}}{\sqrt{n \log n}}\to \sqrt{\frac{2}{k}}$$
\noindent{}in probability as $n\to \infty$. This guarantees that for all  $\epsilon >0$, 
$$\mathbb{P} \left(\left| \frac{\lambda_{\max}(L)}{\sqrt{n \log n}} -\sqrt{\frac{2}{k}} \right| > \epsilon \right) \to 0.$$
\end{proof}

From Proposition \ref{ExtensionBloques} and following similar ideas from
Corollary \ref{corDJ}, we obtain upper and lower bounds for the largest
eigenvalue of a Laplacian block diagonal matrix.

\begin{prop}
\label{TeoDJbloques} Let $L$ be a $k$-block Laplacian diagonal matrix in which
each block is of size $n/k$, where the off-diagonal entries in each block
are independent. Suppose that Condition 1 holds for some $p>6$. If $\mathbb{%
E}[L_{ij}]=0$ and $\text{Var}(L_{ij})=\sigma ^{2}$ for all $n\geq 2$, then,
for all $\epsilon >0$, 

\begin{equation}  \label{cotasuperiormayeig}
\lambda_{\max}(L) \leq \sigma \sqrt{\left(\frac{2}{k}+ \epsilon \right)
n\log n}
\end{equation}

\noindent with high probability, and 
\begin{equation}  \label{cotainferiormayeig}
\lambda_{\max}(L) \geq \sigma \left(2 \sqrt{2} - \sqrt{\frac{2}{k}+ \epsilon}
\right) \sqrt{n \log n}.
\end{equation}
\end{prop}

\begin{proof}
The proof follows from Proposition \ref{ExtensionBloques}, which implies that  

$$\mathbb{P}\left( \left| \frac{\lambda_{\max}(L)}{\sqrt{n \log n}} - \sqrt{\frac{2}{k}}\right| > \delta \right) \to 0 \qquad \text{as} \qquad n \to \infty$$

\noindent{}taking  
$$\delta = \sqrt{\frac{2}{k} + \epsilon} - \sqrt{\frac{2}{k}} >0.$$
\end{proof}

\section{Limiting law for the largest diagonal entry}

Let $L=\left( L_{ij}\right) $ be an $n\times n$ Laplacian matrix whose off-diagonal entries are independent random variables with zero mean and variance one, 
and let 
\begin{equation*}
D^{(n)}=\frac{1}{\sqrt{n-1}}(L_{11},L_{22},\ldots ,L_{nn}).
\end{equation*}%
The random variables $\{D_{i}^{(n)}\}_{i=1,\ldots ,n}$ are not independent
but rather have the covariance matrix 
\begin{equation}
\Sigma _{D^{(n)}}=(\Sigma _{D^{(n)}})_{ij}=\left\{ 
\begin{array}{ccc}
1 & \text{ if } & i=j \\ 
\frac{1}{n-1} & \text{ if } & i\neq j.%
\end{array}%
\right.  \label{covariancematrix}
\end{equation}

To find the limiting law of the largest element of $D^{(n)}$ we first
consider a useful representation for $\Sigma _{D^{(n)}}$,
which holds in the case of any distribution of $L_{ij}$ with the same
covariance matrix.

\noindent{}The eigenvalues of $\Sigma _{D^{(n)}}$  are
$(n-2)/(n-1)$ with multiplicity $n-1$ and $2$ with
multiplicity $1$.

Consider the spectral decomposition of $\Sigma _{D^{(n)}}$ 
\begin{equation*}
\Sigma _{D^{(n)}}=U\Lambda U^{t}
\end{equation*}%
\noindent {}where $\Lambda =(\Lambda _{ii})$ is the diagonal matrix with the
eigenvalues of $\Sigma _{D^{(n)}}$ and $U$ is an orthogonal matrix whose
first $n-1$ columns are an orthonormal basis for the eigenvectors associated to the $n-1$
equal eigenvalues and the last is a normalized eigenvector of the
eigenvalue $2$, namely, $(1/\sqrt{n},\ldots ,1/\sqrt{n})$.

Taking 
\begin{equation}  \label{relacionD}
\tilde{D}^{(n)} = \Sigma_{D^{(n)}}^{-1/2}D^{(n)}
\end{equation}

\noindent{}we have that the covariance matrix of $\tilde{D}^{(n)}$ is the identity
matrix, so the random variables $\{ \tilde{D}_{i}^{(n)}
\}_{i=1}^{ n}$ are uncorrelated.

Then 
\begin{align*}
\left( \Sigma_{D^{(n)}}^{1/2} \right)_{ij} &= \sum_{l=1}^{n}U_{il}
\Lambda_{ll}^{1/2} U_{lj}^{t} \\
&= \sum_{l=1}^{n-1} U_{il}\sqrt{\frac{n-2}{n-1}} U_{lj}^{t}+ \sqrt{2}U_{in}
U_{nj}^{t} \\
&= \sqrt{\frac{n-2}{n-1}} \sum_{l=1}^{n}U_{il}U_{lj}^{t} + \left(\sqrt{2} - 
\sqrt{\frac{n-2}{n-1}} \right) U_{in}U_{nj}^{t},
\end{align*}

\noindent {}and since $U$ is orthogonal and $U_{in}=1/\sqrt{n}$ for $%
i=1,\ldots ,n$, we obtain
\begin{equation}
\left( \Sigma _{D^{(n)}}^{1/2}\right) =\left\{ 
\begin{array}{ccc}
\left( \sqrt{2}-\sqrt{\frac{n-1}{n-2}}\right) \frac{1}{\sqrt{n}}U_{nj}^{t} & 
\text{ if } & i\neq j \\ 
\sqrt{\frac{n-2}{n-1}}+\left( \sqrt{2}-\sqrt{\frac{n-1}{n-2}}\right) \frac{1%
}{\sqrt{n}}U_{ni}^{t} & \text{ if } & i=j. \\ 
&  & 
\end{array}%
\right.
\end{equation}

\noindent{}Thus 
\begin{align}  \label{repreD}
D_i^{(n)} &= \sum_{\substack{ j=1  \\ j\neq i}}^{n} \left(\sqrt{2} -\sqrt{%
\frac{n-1}{n-2}} \right) \frac{1}{\sqrt{n}} U_{nj}^{t}\tilde{D}_{j}^{(n)} +
\left(\sqrt{\frac{n-2}{n-1}} + \left(\sqrt{2} - \sqrt{\frac{n-2}{n-1}}
\right)\frac{1}{\sqrt{n}}U_{nj}^{t} \right) \tilde{D}_{i}^{(n)}  \notag \\
&= \left(\sqrt{2}- \sqrt{\frac{n-2}{n-1}} \right) \frac{1}{\sqrt{n}}
\sum_{j=1}^{n}U_{nj}^{t}\tilde{D}_{j}{(n)} + \sqrt{\frac{n-2}{n-1}} \tilde{D}%
_{i}^{(n)},
\end{align}
\noindent{}which is a useful explicit representation of the elements of $%
D^{(n)}$ in terms of the elements of $\tilde{D}^{(n)}$, $\Lambda$, and $U$.

We now assume that the entries $L_{ij}$ of $L$ are
Gaussian. The following result shows that the Gumbel distribution is the
limiting law of the largest diagonal entry of $L/\sqrt{n-1}$.

\begin{prop}
\label{distribucionmaxdia} Let $L=L_{X}$ be an $n\times n$ Laplacian matrix
with $X=(X_{ij})$ a symmetric matrix whose off-diagonal entries are
independent standard Gaussian random variables. Let 
\begin{equation}
\label{GumbelConstant}
a_{n}=\sqrt{2\log n}\qquad \text{and}\qquad b_{n}=\sqrt{2\log n}-\frac{%
\log \log n+\log 4\pi }{2\sqrt{2\log n}}.
\end{equation}%
Then 
\begin{equation}
M_{n}=a_{n}\left( \frac{\max_{i}L_{ii}}{\sqrt{n-1}}-\sqrt{\frac{n-2}{n-1}}%
b_{n}\right)  \label{convMaxDiago}
\end{equation}%
\noindent {}converges in distribution when $n\rightarrow \infty $ to a
Gumbel random variable.
\end{prop}

\begin{proof}

Since we are assuming Gaussianness,  the uncorrelated random variables $\{D_{i}^*\}_{i=1}^n$ given by \eqref{relacionD} are independent standard Gaussian random variables. Then
$$Z = Z_{n} = \sum_{j=1}^{n} U_{nj}^{t}\tilde{D}_{j}^{(n)}$$
has the standard Gaussian distribution for all $n$ and by \eqref{repreD} 
\begin{equation}\label{DiconDiasterisco}
D_i^{(n)} = \left(\sqrt{2}- \sqrt{\frac{n-2}{n-1}} \right) \frac{1}{\sqrt{n}} Z + \sqrt{\frac{n-2}{n-1}} \tilde{D}_{i}^{(n)}.
\end{equation}

We write 
\begin{align}\label{igualDistri}
a_n \left( \max_{1 \leq i \leq n} D_i^{(n)} - \sqrt{\frac{n-2}{n-1}}  b_n \right)& = a_n \left( \max_{1 \leq i \leq n} \tilde{D}_{i}^{(n)} - b_n \right) \nonumber \\
&+ a_n \left(\max_{1 \leq i \leq n} D_i^{(n)} -\max_{1 \leq i \leq n} \tilde{D}_{i}^{(n)} - \sqrt{\frac{n-2}{n-1}}b_n +b_n\right).
\end{align}

The first term on the right side of \eqref{igualDistri} converges in distribution to the Gumbel law by the classical extreme value result for maxima of i.i.d.\ Gaussian random variables; see \cite{leadbetter2012extremes} or  \cite{resnick1987extreme}. 

On the other hand, using \eqref{DiconDiasterisco} we obtain 
\begin{align*}
\max_{1 \leq i \leq n} D_i^{(n)} &= \max_{1 \leq i \leq n} \left( \left(\sqrt{2}- \sqrt{\frac{n-2}{n-1}} \right) \frac{1}{\sqrt{n}} Z + \sqrt{\frac{n-2}{n-1}} \tilde{D}_{i}^{(n)} \right) \\
&= \left(\sqrt{2}- \sqrt{\frac{n-2}{n-1}} \right)  \frac{1}{\sqrt{n}}Z + \sqrt{\frac{n-2}{n-1}} \max_{1 \leq i \leq  n}\tilde{D}_{i}^{(n)}. 
\end{align*}

\noindent{}Thus 
\begin{align*}
&\left|\max_{1 \leq i \leq n} D_i^{(n)} - \max_{1 \leq i \leq n} \tilde{D}_{i}^{(n)} - \left( \sqrt{\frac{n-2}{n-1}} -1 \right) b_n \right|  \\
&\leq \left|\left(\sqrt{2}- \sqrt{\frac{n-2}{n-1}} \right) \frac{1}{\sqrt{n}}Z  \right| + \left| \sqrt{\frac{n-2}{n-1}} \max_{1 \leq i \leq n}\tilde{D}_{i}^{(n)}- \max_{1 \leq i \leq n} \tilde{D}_{i}^{(n)} - \left( \sqrt{\frac{n-2}{n-1}} -1 \right) b_n \right|\\
&= \left(\sqrt{2}- \sqrt{\frac{n-2}{n-1}} \right)  \frac{1}{\sqrt{n}}\left|Z \right| + \left( \sqrt{\frac{n-2}{n-1}} -1 \right)\left|\max_{1 \leq i \leq n} \tilde{D}_{i}^{(n)} -  b_n \right|.
\end{align*}

Then by Slutsky's theorem 
$$a_n\left|\max_{1 \leq i < n} D_i^{(n)} - \max_{1 \leq i < n} \tilde{D}_{i}^{(n)} - \left( \sqrt{\frac{n-2}{n-1}} -1 \right) b_n \right|  \to 0 \qquad \text{in probability } $$

\noindent{}as $n$ goes to infinity. 
Hence, the second term on the right side of \eqref{igualDistri} also goes to zero in probability as $n$ goes to infinity. 

Then, again using Slutsky's theorem in \eqref{igualDistri}, we obtain that $M_n$ converges in distribution to a Gumbel random variable as $n$ goes to infinity.
\end{proof}

\begin{remark}
The last result is similar to Theorem 3.1 in \cite{berman1964limit}.
However, in \cite{berman1964limit}, a fixed sequence $\{ X_n : n =0 , n = \pm
1, \ldots \}$ of stationary random variables is considered, while
Proposition \ref{distribucionmaxdia} deals with a triangular array $\{
L_{ii}^{(n)} : i =1 , \ldots , n \}$, as is the case in the study of the
largest eigenvalue of ensembles of random matrices.
\end{remark}

\section{Non-asymptotic estimates for the largest eigenvalue}

A consequence of the classical Courant–Fischer min-max Theorem  \cite{horn2012matrix} gives that 
\begin{equation}
\label{min-max_ineq}
\max_{i}L_{ii} \leq \lambda_{\max}(L),
\end{equation} 
without more assumptions than that the matrix $L$ is symmetric. This section establishes a converse inequality, which is valid $\mathrm{w.h.p.}$, for a class of Laplacian matrices of symmetric Wigner matrices. 

The intuition for this comparison is the following.  From Theorem \ref{TeoDJ}(a), $\lambda_{\max}(L)$ grows like $\sqrt{2n \log n}$. On the other hand, the diagonal entries of a Laplacian matrix are sums of independent random variables, so by the Central Limit Theorem they would have an approximately Gaussian distribution for large $n$. It can be proved that for Gaussian random variables $\gamma_i$ (not necessarily independent as is the case of $L_{ii}$) $\max_{i} \gamma_i$ grows like $\sqrt{n \log n}$. This and \eqref{min-max_ineq} suggests that $\max_{i}L_{ii}$ and $\lambda_{\max}(L)$ behave similarly when $n$ is large.  

This comparison of $\max_{i}L_{ii}$ and $\lambda_{\max}(L)$ is known when the Wigner matrix has Gaussian or bounded entries \cite{bandeira2015random,bandeira2020math}. Our Theorem \ref{estimax} below makes rigorous the above motivation, and is a general result under the assumption \eqref{hip2EstiMaxEig} and uses Corollary \ref{corDJ}, which is a consequence of Theorem \ref{TeoDJ}(a). For the sake of completeness we give a proof that this assumption is satisfied when the Laplacian matrix $L$ is constructed from a Wigner matrix with Gaussian entries, using concentration results for the maximum of Gaussian random variables and Sudakov Minorization which deals with correlated Gaussian random variables.

\begin{teo}
\label{estimax} Let $L= (L_{ij})$ be the Laplacian matrix of an $n\times n$
symmetric Wigner matrix $X=(X_{ij})$ constructed from independent random
variables with zero mean, variance $\sigma ^{2}$, and finite $p$-th moment for
some $p>6$. If there exists a $c>0$ such that 
\begin{equation}
\sigma \sqrt{(n-1)\log n}\leq c\max_{i}L_{ii}\qquad \mathrm{w.h.p}.,
\label{hip2EstiMaxEig}
\end{equation}%
\noindent then for all $\epsilon >0$ and 
\begin{equation}
K=c\sqrt{2+\epsilon }>0,
\end{equation}%
\begin{equation}
\lambda _{\max }(L)\leq K\left( 1+\frac{1}{\sqrt{n-1}}\right)
\max_{i}L_{ii}\qquad \mathrm{w.h.p}.  \label{mayoreigenvalorcotasup}
\end{equation}
\end{teo}

\begin{proof}

Let $\hat{L}$ be the matrix given by  
$$\hat{L} = \frac{1}{\sigma} L.$$
\noindent{}We note that for all $1 \leq i < j \leq n$, $\E \hat{L}_{ij} = 0$ and $\E \hat{L}_{ij}^2 = 1$. 

From Corollary \ref{corDJ}, we have that
$$\lambda_{\max} (\hat{L})  \leq \sqrt{(2+\epsilon) n \log n} \qquad \mathrm{w.h.p}.$$
\noindent{}and 
$$\lambda_{\max} (L)  \leq \sigma \sqrt{(2+\epsilon) n \log n}\qquad \mathrm{w.h.p}.$$
\noindent{}Hence, with high probability
\begin{align}\label{desigualdad1}
\lambda_{\max}(L) & \leq \sigma\sqrt{(2+\epsilon) (n-1) \log n} + \sigma \sqrt{(2+\epsilon)  \log n}  \nonumber \\
&\leq c \sqrt{2+ \epsilon}  \max_{i}L_{ii} + \sigma \sqrt{(2+\epsilon) \log n}  \nonumber\\
&=  \left(  c\sqrt{2+ \epsilon} + \frac{1}{\max_{i}L_{ii}} \sigma \sqrt{(2+ \epsilon) \log n} \right) \max_{i} L_{ii}.
\end{align}
From \eqref{hip2EstiMaxEig} we have that
$$\frac{1}{\max_{i}L_{ii}} \sigma \sqrt{(2+ \epsilon) \log n} \leq  \frac{c\sigma \sqrt{(2+ \epsilon) \log n}}{\sigma \sqrt{(n-1) \log n}}= \frac{c\sqrt{2+ \epsilon}}{\sqrt{n-1}}. $$ 
Now, replacing the last term in (\ref{desigualdad1}), we obtain that
\begin{align*}
\lambda_{\max} (L) &\leq \left( c\sqrt{2+ \epsilon} + \frac{c \sqrt{2+ \epsilon}}{\sqrt{n-1}} \right) \max_{i}L_{ii}\\
&=  c \sqrt{2+\epsilon}  \left( 1+ \frac{1}{\sqrt{n-1}} \right) \max_{i} L_{ii}.
\end{align*} 
Hence for $\epsilon>0$, if we take $K=c \sqrt{2+ \epsilon} $, it follows that
\begin{equation*}
\lambda_{\max}(L) \leq K\left( 1+ \frac{1}{\sqrt{n-1}} \right) \max_{i} L_{ii}\qquad \mathrm{w.h.p}.
\end{equation*}
\end{proof}

The inequality \eqref{hip2EstiMaxEig} is satisfied by several distributions. In particular, under Gaussian
assumptions, as we now show.

\begin{prop}
\label{condiciones} Let $L=(L_{ij})$ be the Laplacian matrix of an $n\times n$
symmetric Wigner matrix $X=(X_{ij})$ constructed from i.i.d. standard
Gaussian random variables. Then \eqref{hip2EstiMaxEig} is satisfied.
\end{prop}

\begin{proof}
We note that 
$$\max_{i}L_{ii} = \sqrt{n-1} \max_{i} \frac{L_{ii}}{\sqrt{n-1}} = \sqrt{n-1} \max_{i} \gamma_{i},$$
\noindent{}where $\gamma_{i}$ are standard Gaussian random variables with covariance matrix given by \eqref{covariancematrix}. 

Now, using concentration results (Theorem 3.12 in \cite{massart2007concentration}) we obtain that for all $\alpha >0$
\begin{equation}\label{concentracion}
\max_{i} \gamma_i \geq  \mathbb{E}\max_{i} \gamma_i - \sqrt{\alpha \log n}\qquad \text{w.h.p.}.
\end{equation}
From the structure of the covariance matrix, it follows that for all $i \neq j$ and $n > 3$, 
$$\mathbb{E}\left(\gamma_i- \gamma_j \right)^2 = 2 - 2 \frac{1}{n-1} = 2 \left(\frac{n-2}{n-1} \right) \geq 1.$$
So, Sudakov Minorization (Lemma A.3 in \cite{chatterjee2014superconcentration}) implies
\begin{equation}\label{sudakov}
\mathbb{E} \max_i \gamma_i \geq C \sqrt{\log n},
\end{equation}
where $C$ is a positive constant independent of $n$ and the covariance of the Gaussian random variables $\gamma_i$'s.
Hence, combining \eqref{concentracion} and \eqref{sudakov}, with high probability
\begin{align*}
\max_{i} L_{ii} &\geq \sqrt{n-1} \left(C\sqrt{\log n}- \sqrt{\alpha\log n} \right)\\
&= (C-\sqrt{\alpha}) \sqrt{(n-1)\log n}.
\end{align*}
Then \eqref{hip2EstiMaxEig} is satisfied with $c = (C - \sqrt{\alpha})^{-1}$.
\end{proof}
\begin{remark}

\begin{enumerate}
\item[(a)] It is well known (see for example Lemma 2.3 in \cite%
{massart2007concentration} or Appendix A in \cite%
{chatterjee2014superconcentration}) that the expected value of the maximum
of $n$ standard Gaussian random variables, even if they are not independent,
cannot be much larger that $\sqrt{2\log n}$. For that reason, the universal
constant in the Sudakov Minorization is less than or equal to $\sqrt{2}$. So with this, we can take $\alpha$ such that $C \geq 1- \sqrt{\alpha}$, and taking $\epsilon = 2((C-\sqrt{\alpha})^2-1) \geq 0$ we obtain that in the Gaussian case the constant $K$ can be equal to $\sqrt{2}$. 
\item[(b)] The constant $K$ might not be sharp but it is useful for our weak convergence result in Section 5.
\end{enumerate}

\end{remark}

For the non-Gaussian case, when the random variables $X_{ij}$ have bounded
support, Bandeira proved (proof of Theorem 2.1 in  \cite{bandeira2015random})
that condition \eqref{hip2EstiMaxEig} is satisfied. Moreover, the following bound can be established as a consequence of Theorem 2.1 in \cite{bandeira2015random}, which considers Laplacian matrices constructed from Wigner matrices with bounded entries but not necessarily with the same distribution.

\begin{prop}
Let $L=(L_{ij})$ be the Laplacian matrix of an $n\times n$ symmetric Wigner
matrix $X=(X_{ij})$ constructed from independent random variables bounded by 
$M$, with zero mean and variance $\sigma ^{2}.$ If there exists a $c>0$ such
that 
\begin{equation*}
\sigma \sqrt{n-1}\geq cM\sqrt{\log n},
\end{equation*}%
\noindent {}then there exist positive constants $c_{1},C_{1},\beta $
depending only on $c,$ such that%
\begin{equation}
\lambda _{\max }(L)\leq \left( 1+\frac{C_{1}}{\sqrt{\log n}}\right)
\max_{i}L_{ii}\qquad
\end{equation}%
with probability at least $1-c_{1}n^{-\beta }$.
\end{prop}

\section{Limiting law for the largest eigenvalue}

We are now ready to prove that the Gumbel distribution is the limiting law
of the largest eigenvalue of a Laplacian random matrix constructed from
Gaussian entries. We write 

\begin{equation}
a_{n}'= \frac{a_n}{\sqrt{2}}=\sqrt{\log n}\qquad \text{and}\qquad b_{n}' = \sqrt{2}b_n=2\sqrt{\log n}-\frac{\log
\log n+\log 4\pi }{\sqrt{2\log n}}.  \label{constantesGumbel}
\end{equation}%

\begin{teo}
\label{maintheorem} Let $L=L_{X}$ be an $n\times n$ Laplacian matrix with $%
X=(X_{ij})$ a symmetric matrix whose off-diagonal entries are independent
standard Gaussian random variables. Then 
\begin{equation}
R_{n}=a_{n}'\left( \frac{\lambda _{\max }(L)}{\sqrt{n-1}}-b_{n}'\sqrt{\frac{n-2%
}{n-1}}\right)  \label{resultmaintheorem}
\end{equation}

\noindent{}converges in distribution when $n\to \infty$ to a Gumbel random
variable with distribution $G(x) = \exp (-e^{-x})$ for all $x \in \mathbb{R}$.
\end{teo}

We observe that the rescaling sequences $a_n'$ and $b_n'$ given by %
\eqref{constantesGumbel} have the same order as the appropriate choices for the extreme
value theorem for the maxima of a sequence of Gaussian random variables in
the i.i.d. and the stationary sequence cases; see \cite{berman1964limit}, \cite{resnick1987extreme} 
 respectively and \cite{leadbetter2012extremes}.

\begin{proof}
We first note, with the notation of Section 2,
\begin{align}\label{desigualdadDistribucionlambdamaximo_new}
a_n'\left(\frac{\lambda_{\max}(L)}{\sqrt{n-1}}  - b_n'\sqrt{\frac{n-2}{n-1}} \right) &= a_n \left( \max_{1 \leq i \leq n} D_i^{(n)} -   b_n\sqrt{\frac{n-2}{n-1}} \right) \nonumber \\
&+ a_n' \left( \frac{\lambda_{\max}(L)}{\sqrt{n-1}} - \sqrt{2}\max_{1 \leq i \leq n} D_i^{(n)} \right). 
\end{align}
From Proposition \ref{distribucionmaxdia}, the first term on the right side of  \eqref{desigualdadDistribucionlambdamaximo_new} converges in distribution to a Gumbel random variable. On the other hand, taking $K = \sqrt{2}$ in Theorem \ref{estimax} and Proposition \ref{condiciones}, we have that 
$$\left| \frac{\lambda_{\max}(L)}{\sqrt{n-1}} - \sqrt{2} \max_{1 \leq  i \leq n} D_i^{(n)} \right| \leq  \left| \frac{\sqrt{2}}{\sqrt{n-1}} \max_{1 \leq i \leq n} D_i^{(n)}\right|$$
\noindent{}and
\begin{align}\label{desigualdadesterminodos_new}
\frac{a_n}{\sqrt{2}}\left| \frac{\lambda_{\max}(L)}{\sqrt{n-1}} - \sqrt{2} \max_{1 \leq  i \leq n} D_i^{(n)} \right|  &\leq \frac{a_n}{\sqrt{2}} \left| \frac{\sqrt{2}}{\sqrt{n-1}} \max_{1\leq i \leq n} D_i^{(n)} \right| \nonumber \\
&\leq \frac{a_n}{\sqrt{n-1}}\left| \max_{1\leq i \leq n} D_i^{(n)} - b_n \sqrt{\frac{n-2}{n-1}}  +  b_n \sqrt{\frac{n-2}{n-1}}\right| \nonumber \\
&\leq \frac{a_n}{\sqrt{n-1}} \left| \max_{1\leq i \leq n} D_i^{(n)} - b_n \sqrt{\frac{n-2}{n-1}} \right| \nonumber\\
&\quad \qquad \qquad+ \frac{a_n b_n}{\sqrt{n-1}} \sqrt{\frac{n-2}{n-1}}.
\end{align}
Now using Proposition \ref{distribucionmaxdia} and Slutsky's theorem, the first term in the last inequality goes to zero in probability as $n$ goes to infinity. Hence, using $a_n$ and $b_n$ in \eqref{GumbelConstant}, we have that the second term in \eqref{desigualdadesterminodos_new} goes to zero as $n$ goes to infinity. Therefore
$$\frac{a_n}{\sqrt{2}}\left| \frac{\lambda_{\max}(L)}{\sqrt{n-1}} - \sqrt{2} \max_{1 \leq  i \leq n} D_i^{(n)} \right|  \to  0 \qquad \text{as} \qquad n\to \infty$$
\noindent{}in probability. Thus, the second term of the right side of Equation \eqref{desigualdadDistribucionlambdamaximo_new} tends to zero as $n$ goes to infinity, and we obtain that $R_n$ in \eqref{resultmaintheorem} converges in distribution to a Gumbel random variable as $n$ goes to infinity.
\end{proof}

\begin{remark}
It is easy to find the asymptotic distribution of the largest eigenvalue of
a $k$-block Laplacian diagonal matrix $L$ given by \eqref{blockLap} when
the blocks are independent and all the entries are Gaussian. Indeed the largest
eigenvalue $\lambda _{\max }(L)$ of $L$ satisfies

\begin{equation*}
R_{n}^k=a_{n}'\left( \frac{\lambda _{\max }(L)}{\sqrt{n-1}}-b_{n}'\sqrt{\frac{n-2%
}{n-1}}\right) 
\end{equation*}

\noindent{}converges in distribution when $n\to \infty$ to a random
variable with distribution $G_k(x) = \exp (-k e^{-x})$ for all $x \in \mathbb{R}$, where $a_n'$ and $b_n'$ are given by \eqref{constantesGumbel}.

\end{remark}

\textbf{Acknowledgments.} The authors would like to thank the anonymous referee of our manuscript for the valuable and constructive suggestions that improved the revised version of our work.

\bibliographystyle{siam}
\bibliography{references_clean}

\end{document}